\documentclass[12pt]{amsart}
\usepackage{amsfonts}
\usepackage{bm}
\usepackage{amsfonts,amsmath,amssymb,amscd,bbm,amsthm,mathrsfs,dsfont}
\usepackage{mathrsfs}
\usepackage{pb-diagram}
\usepackage{color}
\usepackage[dvips]{graphicx}
\usepackage{hyperref}
\usepackage{amsfonts}
\usepackage{bm}

\usepackage[latin1]{inputenc}
\usepackage{amsmath}
\usepackage{amsfonts}
\usepackage{amssymb}
\usepackage{graphicx}

\hypersetup{CJKbookmarks=true,      
            colorlinks=true,        
            citecolor=red,         
            linkcolor=blue,         
            urlcolor=blue,          
            bookmarksnumbered=true, 
            bookmarksopen=true,     
            breaklinks=true,        
            pdfstartview=FitH       
           }
\usepackage[all]{xy}

\newtheorem{Thm}{Theorem}[section]
\newtheorem{Lem}[Thm]{Lemma}
\newtheorem{Def}[Thm]{Definition}
\newtheorem{Fac}[Thm]{Fact}
\newtheorem{Cor}[Thm]{Corollary}
\newtheorem{Prop}[Thm]{Proposition}

\newtheorem{Rem}[Thm]{Remark}

\title{Representations of Frobenius-type triangular matrix algebras}

\author{Fang Li $\;\;\;\;\;\;\;\;\;$ Chang Ye}
\address{Fang Li
\newline Department
of Mathematics, Zhejiang University (Yuquan Campus), Hangzhou, Zhejiang
310027, P.R.China}
\email{fangli@zju.edu.cn}

\address{Chang Ye
\newline Department
of Mathematics, Zhejiang University (Yuquan Campus), Hangzhou, Zhejiang
310027,  P.R.China}
\email{yc2009@zju.edu.cn}

\date{version of \today}

\newcommand{\lra}{\longrightarrow}

\newcommand{\ra}{\rightarrow}
\newcommand{\sdp}{\times\kern-.2em\vrule height1.1ex depth-.05ex}
\newcommand{\epi}{\lra \kern-.8em\ra}

\usepackage[left=2.50cm, right=2.50cm, top=3.00cm, bottom=3.00cm]{geometry}

\begin{document}

\renewcommand{\thefootnote}{\alph{footnote}}
\setcounter{footnote}{-1} 
\renewcommand{\thefootnote}{\alph{footnote}}

\begin{abstract}
The aim of this paper is mainly to build a new representation-theoretic realization of finite root systems through the so-called Frobenius-type triangular  matrix algebras by the method of reflection functors over any field. Finally, we give an analog of APR-tilting module for this class of algebras.
 The major conclusions contains the known results as special cases, e.g. that for path algebras over an algebraically closed field \cite{DR2} and for path algebras with relations from symmetrizable cartan matrices \cite{GLS}. Meanwhile, it means the corresponding results for some other important classes of algebras, that is, the path algebras of quivers over Frobenius algebras and the generalized path algebras endowed by Frobenius algebras at vertices.

\end{abstract}
\maketitle
\section{Introduction and Preliminaries}

Let $Q$ be a finite connected acyclic quiver, and let $H=kQ$ be the path algebra of $Q$ for an algebraically closed field $k$. Gabriel showed that the quiver $Q$ is representation-finite if and only if $Q$ is a Dynkin quiver of type $A_n,D_n,E_6,E_7,E_8$ in \cite{Ga}. In this case, there is a bijection between the isomorphism classes of indecomposable representations of $Q$ and the set of positive roots of the corresponding simple complex Lie algebra. Bernstein, Gelfand and Ponomarev introduced the machinery of Coxeter functors, which are defined as compositions of reflection functors, to give an elegant proof of Gabriel's theorem in \cite{BGP}. Gabriel also showed that there are functorial isomorphisms $SC^\pm\cong\tau^\pm(-)$ in \cite{GaAR}, where $S$ is a twist functor and $\tau(-)$ is the Auslander-Reiten translation. Auslander, Platzeck and Reiten showed that there exists an $H$-module $T$ satisfying the functorial isomorphisms $F_k^{\pm}(-)\cong {Hom}_H(T,-)$ in \cite{APR} for a BGP-reflection functor $F_k^{\pm}$ and the APR-tilting module $T$.

Some of these results have been developed to valued graphs or $k$-species by Dlab and Ringel for a field $k$, see \cite{Ringel}, \cite{DR}, \cite{DR2}. Moreover, in \cite{GLS},  for any field $k$, Geiss, Leclerc and Schr{\"o}er generalized them to a class of 1-Gorenstein algebras $A$, which were defined via quivers with relations associated to symmetrizable Cartan matrices, as follows.

  Let $C=(c_{ij})\in M_n(\mathbb{Z})$ be a symmetrizable generalized Cartan matrix with a symmetrizer  $D={diag}(c_1,\cdots,c_n)$. For all $c_{ij}<0$, write that
 $g_{ij}:=|gcd(c_{ij},c_{ji})|,~~~~f_{ij}:=|c_{ij}|/g_{ij},~~~~k_{ij}:=gcd(c_i,c_j).$

An {\em orientation} of $C$ is a subset $\Omega\subset \{1,2,\ldots,n\}\times \{1,2,\ldots,n\}$ such that the following hold:

 $(i)~\{(i,j),(j,i)\}\cap \Omega\neq\emptyset$ if and only if $c_{ij}<0$;

$(ii)$For each sequence $((i_1,i_2),(i_2,i_3),\ldots,(i_t,i_{t+1}))$ with $t\geq 1$ and $(i_s,i_{s+1})\in\Omega$ for all $1\leq s\leq t$ we have $i_1\neq i_{t+1}$.\\
For an orientation $\Omega$ of $C$, let $Q:=Q(C,\Omega):=(Q_0,Q_1,s,t)$ be the quiver with the set of vertices $Q_0:={1,\ldots,n}$ and the set of arrows
$$Q_1:=\{\alpha_{ij}^{(g)}:j\rightarrow i|(i,j)\in \Omega, 1\leq g\leq g_{ij}\}\cup \{\varepsilon_i:i\rightarrow i|1\leq i\leq n\}.$$

For a qiuver $Q=Q(C,\Omega)$ and a symmetrizer $D={diag}(c_1,\ldots,c_n)$ of $C$, let
\begin{equation}\label{Goren}
A=A(C,D,\Omega):=kQ/I
\end{equation}
where $kQ$ is the path algebra of $Q$, and $I$ is the ideal of $kQ$ defined by the following relations:

$(i)$ For each $i$ we have the nilpotency relation $\varepsilon_i^{c_i}=0$.

$(ii)$ For each $(i,j)\in\Omega$ and each $1\leq g\leq g_{ij}$ we have the commutativity relation $\varepsilon_i^{f_{ji}}\alpha_{ij}^{(g)}=\alpha_{ij}^{(g)}\varepsilon_j^{f_{ij}}.$

It was proved in \cite{GLS} that the algebra $A$, given in (\ref{Goren}) is of 1-Gorenstein.

The aim of this paper is to give a more larger class of 1-Gorenstein algebras in which the important conclusions about representation theory still hold.

 In the sequel, the ground field $k$ is always permitted to be any field.

 For $n\geq 2$, define a \emph{triangular matrix algebra} $\Gamma$ of order $n$ satisfying $\Gamma=\left(
                                                                          \begin{array}{cccc}
                                                                            A_1 & A_{12} & \ldots& A_{1n} \\
                                                                            0 & A_2& \ldots & A_{2n} \\
                                                                             \vdots& \vdots& \ddots&
                                                                             \vdots \\
                                                                             0& 0& \cdots&
                                                                             A_n \\
                                                                          \end{array}
                                                                        \right)$, where each $A_i$ is an algebra and  $A_{ij}$  an $A_i$-$A_j$-bimodule with bimodule maps $\mu_{ilj}:A_{il}\otimes_{A_l}A_{lj}\rightarrow A_{ij}$ such that the following diagram commutes:
                                                                        \begin{equation*}
\CD
  A_{il}\otimes_{A_l}A_{lj}\otimes_{A_j}A_{jt} @>\mu_{ilj}\otimes {id}_{A_{jt}}>> A_{ij}\otimes_{A_j}A_{jt} \\
  @V {id}_{A_{il}}\otimes \mu_{ljt} VV @VV    \mu_{ijt}V  \\
  A_{il}\otimes_{A_l}A_{lt} @>\mu_{ilt} >> A_{it}
\endCD
\end{equation*}
for $1\leq i<l<j<t\leq n$, whose multiplication is given by $(AB)_{ij}=\sum_{i<l<j} \mu_{ilj}(a_{il}\otimes b_{lj})$ for $A=(a_{ij})_{n\times n}, B=(b_{ij})_{n\times n}\in \Gamma$ where $(AB)_{ij}$ means the $(i,j)$-entry of $AB$.

A {\em representation} $X$ of $\Gamma$ is defined as a datum $X=\left(
       \begin{array}{c}
        X_1 \\
         X_2\\
         \vdots \\
         X_n\\
       \end{array}
     \right)_{\phi_{ij}}$, with $\phi_{ij}$: $A_{ij}\otimes_{A_j}X_j\rightarrow X_i$ an $A_i$-module morphism for $1\leq i<j\leq n$ so that it satisfies the following commutative diagram:\begin{equation*}
\CD
  A_{ij}\otimes_{A_j}A_{jq}\otimes      _{A_q}X_q @>\mu_{ijq}\otimes {id}_{X_q}>> A_{iq}\otimes_{A_q}X_q \\
  @V {id}_{A_{ij}}\otimes \phi_{jq} VV @VV    \phi_{iq}V  \\
  A_{ij}\otimes_{A_j}X_j @>\phi_{ij} >> X_i
\endCD
\end{equation*}

A {\em morphism} $f$ from a representation $X$ to another representation $Y$ is defined as a datum $(f_i,1\leq i\leq n)$, where $f_i: X_i\rightarrow Y_i$ is an $A_i$-map such that $f_i\phi_{ij}=\phi_{ij}(id_{A_{i,j}}\otimes f_j)$ for each $1\leq i<j\leq n$.

Then we obtain the {\em representation category} of $\Gamma$, denoted as ${Rep}(\Gamma)$.

\begin{Def}\label{frobenius type}
Let $A_i$ be  Frobenius algebras with units $e_i$ for all $i=1,\cdots,n$. Let $B_{ij}$ be $A_i$-$A_j$-bimodules  as free left $A_i-$modules and free right $A_j-$modules of finite ranks respectively satisfying ${ Hom}_{A_i}(B_{ij},A_i)\cong { Hom}_{A_j}(B_{ij},A_j)$ for all $1\leq i<j\leq n$. Let
\begin{equation}\label{grade}
{A_{ij}=\oplus_{l=0}^{j-i-1}\oplus_{i<k_1<k_2<\cdots<k_l<j}B_{ik_1}\otimes_{A_{k_1}}B_{k_1k_2}\otimes\cdots\otimes _{A_{k_l}}B_{k_lj}
}
\end{equation}
  for $1\leq i<j\leq n$, where $l=0$ means the direct summand $B_{ij}$.

Define a triangular matrix algebra $\Lambda=\left(
                                                                          \begin{array}{cccc}
                                                                            A_1 & A_{12} & \ldots& A_{1n} \\
                                                                            0 & A_2& \ldots & A_{2n} \\
                                                                             \vdots& \vdots& \ddots&
                                                                             \vdots \\
                                                                             0& 0& \cdots&
                                                                             A_n \\
                                                                          \end{array}
                                                                        \right)$
                                                                         which is called a {\em Frobenius-type triangular matrix algebra}, if it satisfies (\ref{grade}) where the map $\mu_{ijq}:A_{ij}\otimes_{A_j}A_{jq}\rightarrow A_{iq}$ is the natural inclusion                                                                 map.
\end{Def}

 The algebra $A=A(C,D,\Omega)=kQ/I$ in (\ref{Goren}) from \cite{GLS} is indeed a Frobenius-type triangular matrix algebra. To see this, one needs only to  re-order vertexes by $i<j$ if $(i,j)\in \Omega$ and take $A_i=e_iHe_i,~B_{ij}=A_i{Span}_k(\alpha_{ij}^{(g)}|1\leq g\leq g_{ij})A_j$. It is easy to see such  $A$ satisfy the condition of Frobenius-type triangular matrix algebra in Definition \ref{frobenius type}.

 Following this fact, in this paper the major results on Frobenius-type triangular matrix algebras are the improvement of the corresponding ones in \cite{GLS}.
For convenience, in the sequel, we will always assume $A_i$ to be finite dimensional for all $i=1,\cdots,n$.

\begin{Rem}\label{rem1.2}
$(i)$ In Definition \ref{frobenius type}, if each $A_i$ is only a finite dimension algebra, and the condition ${ Hom}_{A_i}(B_{ij},A_i)\cong { Hom}_{A_j}(B_{ij},A_j)$ is replaced by that ${ Hom}_{A_i}(B_{ij},A_i)$ is a projective $A_j$-module or an injective $A_j$-module, then $\Lambda$ is called a {\em normally upper triangular gm algebra}, which was introduced and investigated in \cite{LY}.

Since ${ Hom}_{A_j}(B_{ij},A_j)$ is a free left $A_j$-module and ${ Hom}_{A_i}(B_{ij},A_i)\cong { Hom}_{A_j}(B_{ij},A_j)$ we obtain that ${ Hom}_{A_i}(B_{ij},A_i)$ is a free left $A_j$-module, then, of course, a projective $A_j$-module. So all Frobenius-type triangular matrix algebras are normally upper triangular gm algebras.

$(ii)$ In Definition \ref{frobenius type}, let $B_{ij}=0$ for all $|i-j|\geq 2$ and but $A_i$ are allowed to be any rings, then we obtain a class of triangular matrix rings
 which were studied in \cite{SL} about their module category and some homological characterizations.

$(iii)$  A path algebra of quiver over algebra $\Lambda=AQ=A\otimes_k kQ$ for an acyclic quiver $Q$ and a Frobenius algebra $A$ which studied in \cite{RZ} is also a Frobenius-type triangular matrix algebra. Here one needs only to take all $A_i=A$ and $B_{ij}=\oplus_{s=1}^{\#\{\alpha:j\rightarrow i\}} A$. In particular, it was investigated in \cite{RZ} for the case $A=k[T]/[T^2]$ the algebra of dual numbers.

$(iv)$  A generalized path algebra $\Lambda=k(Q,\mathcal A)$, with an acyclic quiver $Q$, $\mathcal A=\{A_i\}_{i\in Q_0}$ and Forbenius algebras $A_i (i\in Q_0)$, is a Frobenius-type triangular matrix algebra through taking all $B_{ij}$ to be generalized arrows from $j$ to $i$ and thus obtaining $A_{ij}$ as  generalized paths from $j$ to $i$. For more details, see \cite{LY}.

\end{Rem}

 Let $X=\left(
       \begin{array}{c}
        X_1 \\
         X_2\\
         \vdots \\
         X_n\\
       \end{array}
     \right)_{\phi_{ij}}$ be a $\Lambda$-module, it is easy to see that $\phi_{ij}$: $A_{ij}\otimes_{A_j}X_j\rightarrow X_i$ is uniquely determined by $\phi_{ij}|_{{Res}(B_{ij}\otimes X_j)}:B_{ij}\otimes_{A_j}X_j\rightarrow X_i$ for $1\leq i<j\leq n$.

\begin{Rem}
{Without ambiguity, we sometimes omit to write $\phi_{ij}$, especially when $\phi_{ij}$ is a natural inclusion for all $1\leq i<j\leq n$.
}
\end{Rem}
\begin{Rem}
In this paper, we consider only $\Lambda$ which is connected, that is, $\Lambda$ can NOT be written as $\left(
                                                                                                 \begin{array}{cc}
                                                                                                   \Lambda_1 & 0 \\
                                                                                                   0 & \Lambda_2\\
                                                                                                 \end{array}
                                                                                               \right)$  for two non-zero Frobenius-type triangular matrix algebras $\Lambda_1$ and $\Lambda_2$.
                                                                                                Otherwise, then $\Lambda=\left(
                                                                                                 \begin{array}{cc}
                                                                                                   \Lambda_1 & 0 \\
                                                                                                   0 & \Lambda_2\\
                                                                                                 \end{array}
                                                                                               \right)$, whose any representation $M$ is always a direct sum  representation $M_i$ of $\Lambda_i$ ($i=1,2$), that is, $M=M_1\oplus M_2$.
\end{Rem}

This article is organized as follows. In Section 2, we study some properties of Frobenius-type triangular matrix algebras and in particular, show that they are a class of 1-Gorenstein algebras. In Section 3, we define the reflection functors for Frobenius-type triangular matrix algebras and give the relation between Coxeter functors and Auslander-Reiten translation for these algebras. In Section 4, we recall some definitions and basic facts on Cartan matrices, quadratic forms and Weyl groups. Then we give a new representation-theoretic realizations of all finite root systems via Frobenius-type triangular matrix algebras. In Section 5, we study the generalized versions of $APR$-tilting modules over Frobenius-type triangular matrix algebras.

\section{{ Locally free modules and 1-Gorenstein property }}

\begin{Prop}\cite{LY}\label{representation}
{The representation category of a Frobenius-type triangular matrix algebra $\Lambda$ and the module category of $\Lambda$ are equivalent.
}
\end{Prop}

We denote $B_{ji}={ Hom}_{A_i}(B_{ij},A_i)\cong { Hom}_{A_j}(B_{ij},A_j)$ for $1\leq i<j\leq n$.

\begin{Prop}
{For $1\leq i<j\leq n,~B_{ji}$ is a free left $A_j-$module and a free right $A_i-$module. Also, we have  $B_{ij}={ Hom}_{A_i}(B_{ji},A_i)\cong { Hom}_{A_j}(B_{ji},A_j)$.
}
\end{Prop}

\begin{proof}
It is trivial.
\end{proof}
\begin{Prop}\label{isomorphism}
For any $A_j$-module $M_j$ and $A_i$-module $N_i$, we have the isomorphism $${ Hom}_{A_i}(B_{ij}\otimes M_j,N_i)\cong { Hom}_{A_j}(M_j,B_{ji}\otimes N_i).$$
\end{Prop}

\begin{proof}
  The adjunction map gives an isomorphism of $k$-vector spaces:
  $$ { Hom}_{A_i}(B_{ij}\otimes M_j,N_i)\cong { Hom}_{A_j}(M_j,{ Hom}_{A_i}(B_{ij},N_i)).$$
  Since ${ Hom}_{A_i}(B_{ij},N_i)\cong { Hom}_{A_i}(B_{ij},A_i\otimes_{A_i}N_i)\cong { Hom}_{A_i}(B_{ij},A_i)\otimes_{A_i}N_i\cong B_{ji}\otimes_{A_i}N_i$,\\
  we have that: ${ Hom}_{A_i}(B_{ij}\otimes M_j,N_i)\cong { Hom}_{A_j}(M_j,B_{ji}\otimes N_i)$.
\end{proof}
\begin{Lem}\label{dualbasis}
\rm{\textbf{[The Dual Basis Lemma]}} Let $A$ be an Artin algebra, and $P$ be a finite generated projective $A$-module, ${Hom}_A(P,A)$, there exist $x_1,\cdots,x_m\in P,~f_1,\cdots,f_m\in {Hom}_A(P,A),$ such that for each $x\in P, x=\Sigma_{i=1}^m f_i(x)x_i$.
\end{Lem}

Since $B_{ij}$ is finite generated projective left $A_i$-module and finite generated projective right $A_j$-module, by the above lemma, there exist $L_{ij}\subseteq B_{ij},~L_{ij}^*\subseteq  {Hom}_{A_i}(B_{ij},A_i)$ and $R_{ij}\subseteq B_{ij},~R_{ij}^*\subseteq  {Hom}_{A_j}(B_{ij},A_j)$ such that for each $b_{ij}\in B_{ij},~b_{ij}=\Sigma_{\ell\in L_{ij}} \ell^*(b_{ij})\ell=\Sigma_{r\in R_{ij}} r^*(b_{ij})r$.

For the isomorphism in Proposition \ref{isomorphism}, for each $\phi_{ij}\in { Hom}_{A_i}(B_{ij}\otimes M_j,N_i)$, we have
$\overline{\phi_{ij}} \in { Hom}_{A_j}(M_j,B_{ji}\otimes N_i)$ satisfying $\overline{\phi_{ij}}(m_j)= \Sigma_{\ell\in L_{ij}}\ell^*\otimes \phi_{ij}(\ell\otimes m_j).$

On the other hand, for each $\psi_{ij}\in { Hom}_{A_j}(M_j,B_{ji}\otimes N_i)$, we have
$\overline{\psi_{ij}} \in { Hom}_{A_i}(B_{ij}\otimes M_j,N_i)$ satisfying $\overline{\psi_{ij}}(b_{ij}\otimes m_j)= \Sigma_{\ell\in L_{ij}}\ell^*(b_{ij})\psi_{ij}(m_j)_\ell,$
where the elements $\psi_{ij}(m_j)_l\in N_i$ are uniquely determined by $\psi_{ij}(m_j)=\Sigma_{\ell\in L_{ij}}\ell^*\otimes\psi_{ij}(m_j)_\ell$.

Let $\Lambda$ be a Frobenius-type triangular matrix algebra as in Definition \ref{frobenius type}. Denote $P_i=\Lambda e_i$ the projective $\Lambda$-module for the idempotent $e_i$ as the unit 1 of $A_i$ and $I_i$ the corresponding injective $\Lambda$-module for $1\leq i \leq n$. Obviously,
     \begin{equation}\label{proj}
P_i=(A_{1i}\; A_{2i}\; \cdots \; A_{ii}\; 0\; \cdots\; 0)^t.
     \end{equation}
      Also, we denote ${E_i}$ the $\Lambda$-module $(0\; \cdots\; A_i\; \cdots \; 0)^t$ where $A_i$ is in the $i-th$ row.

\begin{Lem}\label{tensor projective}
\cite{LY} For two $k$-algebras $A$ and $B$, assume that $M$ is an $A$-$B$-bimodule such that $M$ is a projective left $A$-module, $P$ is a projective left $B$-module. Then $M\otimes_BP$ is a projective left $A$-module.
\end{Lem}
\begin{Prop}\label{dimension1}
{For every $i=1,\ldots,n$,  ${proj.dim}(E_i)\leq 1$ and ${inj.dim}(E_i)\leq 1.$
}
\end{Prop}
\begin{proof}
Clearly, $E_1=P_1$.
For every $i=2,\ldots,n,$ we have exact sequences:
\begin{equation}\label{rankproj}
0\rightarrow \oplus_{j=1}^{i-1}P_j\otimes_{A_j}B_{ji}\rightarrow P_i\rightarrow E_i\rightarrow 0.
\end{equation}
So, ${proj.dim}(E_i)\leq 1$ by Lemma \ref{tensor projective}.
Let $E'_i$ be the right $\Lambda$-module such that $D(E'_i)\cong E_i$. For $i=1,\cdots,n-1$, there is a canonical exact sequence
\begin{equation}\label{iprojresolution}
0\rightarrow \oplus_{j=i+1}^{n}B_{ij}\otimes_{A_j}e_j\Lambda\rightarrow e_i\Lambda\rightarrow E'_i\rightarrow 0.
\end{equation}
Applying the duality $D$ to (\ref{iprojresolution}) we get a minimal injective resolution
\begin{equation}\label{injresolution}
0\rightarrow E_i\rightarrow D(e_i\Lambda)\rightarrow\oplus_{j=i+1}^{n}D(B_{ij}\otimes_{A_j}e_j\Lambda) \rightarrow 0.
\end{equation}
Clearly, $E_n$ is injective.
So, ${inj.dim}(E_i)\leq 1$ for $1\leq i\leq n$.
\end{proof}

\begin{Def}
 For a Frobenius-type triangular matrix algebra $\Lambda$, following \cite{GLS}, a finitely generated $\Lambda$-representation $X$ is called \emph{locally free} if $X_i$ are free $A_i$-modules for all $1\leq i\leq n$.

 Denote by ${rep}_{l.f.}(\Lambda)$ the subcategory of all locally free $\Lambda$-modules.
\end{Def}

\begin{Cor}\label{projdim}
For a Frobenius-type triangular matrix algebra $\Lambda$, if $X\in  rep_{l.f.}(\Lambda)$, then $ proj.dim(X)\leq 1,~ inj.dim(X)\leq 1.$
\end{Cor}
\begin{proof}
We have a short exact sequence
$$0\rightarrow e_1X\rightarrow X\rightarrow (1-e_1)X\rightarrow 0.$$
where $e_1X$ and $(1-e_1)X$ are locally free. By Proposition \ref{dimension1} and using induction on $n$, we know that the projective and the injective dimensions of $e_1X$ and $(1-e_1)X$ are at most one.
\end{proof}
\begin{Rem}
For the algebra $A=A(C,D,\Omega)$ in \cite{GLS}, the three conditions $$X\in  rep_{l.f.}(A),~\;\;~~proj.dim(X)\leq 1,~~~\;\; inj.dim(X)\leq 1$$ are equivalent. But it is generally not true for a Frobenius-type triangular matrix algebra. So, Corollary \ref{projdim} is available for the sequel.
\end{Rem}

An algebra $A$ is called \emph{m-Gorenstein} if
$$ inj.dim(A)\leq m ~~~\; {and} \;\;~~~proj.dim(DA)\leq m.$$ Such algebras were firstly introduced and studied in \cite{IY}.

\begin{Cor}
The Frobenius-type triangular matrix algebra $\Lambda$ is a 1-Gorenstein algebra.
\end{Cor}
\begin{proof}
It is a direct consequence of Corollary \ref{projdim}.
\end{proof}

Recall that for an algebra $A$ an $A$-module $X$ is $\tau$-rigid (resp. $\tau^-$-rigid) if ${Hom}_A(X,\tau(X))=0$ (resp. ${Hom}_A(\tau^-(X),X)=0$).
 We call $\Lambda$-mod $X$  \emph{rigid} if ${Ext}_\Lambda^1(X,X)=0$.

 When ${proj.dim}(X)\leq 1$, we have a functorial isomorphism ${Ext}_A^1(X,Y)\cong D{Hom}_A(Y,\tau(X));$
 when ${inj.dim}(X)\leq 1$, we have a functorial isomorphism ${Ext}_A^1(X,Y)\cong D{Hom}_A(\tau^-(Y),X).$

\begin{Cor}\label{equal}
For a Frobenius-type triangular matrix algebra $\Lambda$, let $X\in  rep_{l.f.}(\Lambda)$. Then,  $X$ is rigid if and only if $X$ is $\tau$-rigid, also if and only if $X$ is $\tau^-$-rigid.
\end{Cor}
\begin{Prop}\label{closed}
For a Frobenius-type triangular matrix algebra $\Lambda$, the subcategory ${rep}_{l.f.}(\Lambda)$ is closed under extensions, kernels of epimorphisms and cokernels of monomorphisms.
\end{Prop}

\begin{proof}
Let $0\rightarrow X \stackrel{f}\rightarrow Y \stackrel{g}\rightarrow Z\rightarrow 0$
be a short exact sequence in ${rep}(\Lambda)$. For each $1\leq i\leq n$, this induces a short exact sequence
\begin{equation}\label{lf}
0\rightarrow e_iX \stackrel{f}\rightarrow e_iY \stackrel{g}\rightarrow e_iZ\rightarrow 0
\end{equation}
of left $A_i$-modules.

 By the definition, when $X, Z\in {rep}_{l.f.}(\Lambda)$,  $e_iX, e_iZ$ are both $A_i$-free and then $e_iX$ is injective via $A_i$ is a Frobenius algebra and $e_iZ$ is projective  for any $i$. Hence, the short exact sequence (\ref{lf}) splits, i.e. $e_iY\cong e_iX\oplus e_iZ$.
It follows that each $e_iY$ is free $A_i$-module and then $Y\in{rep}_{l.f.}(\Lambda)$.

When $Y,Z\in{rep}_{l.f.}(\Lambda)$, $e_iY, e_iZ$ are both $A_i$-free and then $e_iZ$ is a projective module. Hence (\ref{lf}) splits, i.e. $e_iY\cong e_iX\oplus e_iZ$. By the Krull-Schmidt theorem, it is easy to see that each $e_iX$ is a free $A_i$-module. So, $X\in{rep}_{l.f.}(\Lambda)$.

  When $X,Y\in{rep}_{l.f.}(\Lambda)$, $e_iX, e_iY$ are both $A_i$-free and then $e_iX$ is injective via $A_i$ is a Frobenius algebra. Hence (\ref{lf}) splits. Similarly, by the Krull-Schmidt theorem, each $e_iZ$ is free $A_i$-module. So, $Z\in{rep}_{l.f.}(\Lambda)$ .
\end{proof}

\section{{Reflection functors and AR-translation}}

Let $\Lambda=\left(
                                                                          \begin{array}{cccc}
                                                                            A_1 & A_{12} & \ldots& A_{1n} \\
                                                                            0 & A_2& \ldots & A_{2n} \\
                                                                             \vdots& \vdots& \ddots&
                                                                             \vdots \\
                                                                             0& 0& \cdots&
                                                                             A_n \\
                                                                          \end{array}
                                                                        \right)$ be a Frobenius-type triangular matrix algebra. Denote $S_1(\Lambda)=\left(
                                                                          \begin{array}{ccccc}
                                                                            A_2 & A_{23} & \ldots& A_{2n}& A_{21} \\
                                                                            0 & A_3& \ldots & A_{3n} & A_{31} \\
                                                                             \vdots& \vdots& \ddots&  \vdots & \vdots\\
                                                                             0&0&\cdots& A_n & A_{n1} \\
                                                                             0&0&\cdots& 0 & A_1\\
                                                                          \end{array}
                                                                        \right)$, where $B_{j1}={ Hom}_{A_1}(B_{1j},A_1)$ and $${A_{j1}=\oplus_{l=0}^{n-j}\oplus_{j<k_1<k_2<\cdots<k_l\leq n}B_{jk_1}\otimes_{A_{k_1}}B_{k_1k_2}\otimes\cdots\otimes _{A_{k_l}}B_{k_l 1}
                                                                        }$$ for $j=2,\ldots, n.$

   $B_{j1}$ is a free right $A_1$-module and also a free left $A_j$-module, since $B_{j1}={ Hom}_{A_1}(B_{1j},A_1)\cong { Hom}_{A_j}(B_{1j},A_j)$.
      For $S_1(\Lambda)$, we have

   ${ Hom}_{A_1}(B_{j1},A_1)\cong { Hom}_{A_1}({ Hom}_{A_1}(B_{1j},A_1),A_1)\cong B_{1j}\cong  { Hom}_{A_j}({ Hom}_{A_j}(B_{1j},A_j),A_j)\cong{ Hom}_{A_j}(B_{j1},A_j)$,
  \\
   which means that $S_1(\Lambda)$ is still a Frobenius-type triangular matrix algebra.

Using the same method in sequence, we can obtain the Frobenius-type triangular matrix algebra $S_k S_{k-1}\ldots S_1(\Lambda)$ for any $k$. In particular, it can be seen that $\Lambda\cong S_n S_{n-1}\ldots S_1(\Lambda)$.

A {\em reflection functor} $F_1^+: rep(\Lambda)\rightarrow  rep(s_1(\Lambda))$ can be described as follows.

For $X=\left(
       \begin{array}{c}
        X_1 \\
         X_2\\
         \vdots \\
         X_n\\
       \end{array}
     \right)_{\phi_{ij}}\in  rep(\Lambda)$, define $ F_1^+(X)=\left(
            \begin{array}{c}
             X_2 \\
              \vdots \\
              X_n\\
              X_1'\\
            \end{array}
          \right)_{\phi^{'}_{ij}}$, with $X_1^{'}= Ker(X_{1,in}),$ where
          $$X_{1,in}:\oplus_{k=2}^n B_{1k}\otimes X_k\rightarrow X_1$$
Denote by $\psi_{i1}$ the composition of the inclusion map $X_1'\hookrightarrow \oplus_{k=2}^n B_{1k}\otimes X_k$ and the projection $\oplus_{k=2}^n B_{1k}\otimes X_k \twoheadrightarrow B_{1i}\otimes X_i$. Then $$\phi^{'}_{ij}=\left\{
  \begin{array}{ll}
    \overline{\psi_{i1}},  & \hbox{for $j=1$;} \\
    \phi_{ij}, & \hbox{otherwise.}
  \end{array}
\right.$$

For a morphism $f=\{f_i\}:X\rightarrow Y$ in $rep(\Lambda)$, $F_1^+(f)=f'=\{f'_i\}$ where $f'_i=f_i$ for $i=2,\ldots,n$ and $f'_1$ is the unique morphism make the following diagram commutes:
$$\xymatrix@C=0.5cm{
 0\ar[r]&Ker(X_{1,in})\ar@{-->}[d]^{f'_1} \ar[r] & \oplus_{k=2}^{n} B_{1k}\otimes X_k \ar[r]^{X_1,in}\ar[d]^{\oplus id\otimes f_k} & X_1\ar[d]^{f_1}\\
 0\ar[r]&{Ker}(Y_{1,in})\ar[r] & \oplus_{k=2}^{n} B_{1k}\otimes Y_k\ar[r]^-{Y_{1,in}} & Y_1
  }$$

Similarly, for $X=\left(
       \begin{array}{c}
        X_2 \\
         \vdots \\
         X_n\\
         X_1\\
       \end{array}
     \right)_{\phi_{ij}}\in  rep(S_1(\Lambda))$,  define $F_1^-: rep(S_1(\Lambda))\rightarrow  rep(\Lambda)$ satisfying
 $F_1^-(X)=\left(
            \begin{array}{c}
             X_1'\\
             X_2 \\
             \vdots \\
             X_n\\
            \end{array}
          \right)_{\phi^{'}_{ij}}$, with $X_1^{'}={Coker}(X_{1,out}),$ where
          $X_{1,out}:=(\overline{\phi_{j1}})_j:X_1\rightarrow\oplus_{j=2}^n B_{1j}\otimes X_j$.

  Denote by $\psi_{1j}$ the composition of the inclusion map $B_{1j}\otimes X_j\hookrightarrow \oplus_{k=2}^n B_{1k}\otimes X_k$ and the projection $\oplus_{k=2}^n B_{1k}\otimes X_k \twoheadrightarrow X_1'$. Then $$\phi^{'}_{ij}=\left\{
  \begin{array}{ll}
    \psi_{1j}, & \hbox{if $i=1$;} \\
    \phi_{ij}, & \hbox{otherwise.}
  \end{array}
\right.$$

For a morphism $g=\{g_i\}:X\rightarrow Y$ in $rep(S_1(\Lambda))$, $F_1^-(g)=g'=\{g'_i\}$ where $g'_i=g_i$ for $i=2,\ldots,n$ and $g'_1$ is the unique morphism make the following diagram commutes:

$$\xymatrix@C=0.5cm{
  X_1\ar[d]^{g_1} \ar[r]^-{X_{1,out}} & \oplus_{k=2}^{n} B_{1k}\otimes X_k \ar[r]\ar[d]^{\oplus id\otimes g_k} & {Cok}(X_{1,out})\ar@{-->}[d]^{g'_1}\ar[r]&0\\
 Y_1 \ar[r]^-{Y_{1,out}}& \oplus_{k=2}^{n} B_{1k}\otimes Y_k\ar[r] & {Cok}(Y_{1,out})\ar[r]&0
  }$$

In the same way in sequence, for any $k$, we can define $$F_k^+: rep(S_{k-1} S_{k-2}\ldots S_1(\Lambda))\rightarrow  rep(S_k S_{k-1}\ldots S_1(\Lambda)),$$ $$F_k^-: rep(S_k S_{k-1}\ldots S_1(\Lambda))\rightarrow  rep(S_{k-1} S_{k-2}\ldots S_1(\Lambda)).$$

 Then we denote $$C^+=F_n^+F_{n-1}^+\ldots F_1^+: rep(\Lambda)\rightarrow  rep(\Lambda)\;\;\; and \;\;\; C^-=F_1^-F_2^-\ldots F_n^-: rep(\Lambda)\rightarrow  rep(\Lambda)$$ which are called \emph{the Coxeter functors} on ${rep}(\Lambda).$

\begin{Rem}  Let $Q$ be a connected acyclic quiver. We can re-arrange the vertices of $Q$ by making $i<j$ if there is a path from $j$ to $i$ so as to get an admissible sequence $1, 2, \cdots, n$ with $1$ as a  sink vertex and $n$ as a source vertex.  Thus, the path algebra $kQ$ can be seen as a special case of Frobenius-type triangular matrix algebras through assuming $A_i=k$ for all $i$ and $B_{ij}$ is the $k$-linear spaces generated by all arrows from $j$ to $i$.  In this case, the reflection functors and Coxeter functors defined above accords with that as the classical case given by Dlab and Ringel in \cite{DR}\cite{DR2}\cite{Ringel}.
\end{Rem}
\begin{Prop}\label{adj}
For a Frobenius-type triangular matrix algebra $\Lambda$, and $X\in {rep}(S_1(\Lambda)),~ Y\in {rep}(\Lambda),$ there is a functorial isomorphism
$${Hom}_\Lambda(F_1^-(X),Y)\cong{Hom}_{S_1(\Lambda)}(X,F_1^+(Y)).$$
\end{Prop}
\begin{proof}
Consider a morphism $f\in {Hom}_{S_1(\Lambda)}(X,F_1^+(Y))$. By definition, this is a collection of $A_j$-module homomorphisms
$f_j:X_j\rightarrow (F_1^+(Y))_j$
with $1\leq j\leq n$ satisfying certain commutative relations.
In the diagram
$$\xymatrix@C=0.5cm{
  &X_1\ar[d]^{f_1} \ar[r]^-{X_{1,out}} & \oplus_{k=2}^{n} B_{1k}\otimes X_k \ar[r]\ar[d]^{\oplus id\otimes f_k} & {Cok}(X_{1,out})\ar@{-->}[d]^{g_1}\ar[r]&0\\
 0\ar[r]&{Ker}(Y_{1,in})\ar[r] & \oplus_{k=2}^{n} B_{1k}\otimes Y_k\ar[r]^-{Y_{1,in}} & Y_1&
  }$$
its left square commutes. Observe that $Y_{1,in}\circ (\oplus id\otimes f_k) \circ X_{1,out}=0$, so $Y_{1,in}\circ (\oplus id\otimes f_k)$ factors through the cokernel of $X_{1,out}$. Then, there is a map $g_1$ such that the right square commutes.

Thus if we set $g_j:=f_j$ for $2\leq j\leq n$, we get a homomorphism $g:F_1^-(X)\rightarrow Y$ corresponding to the given $f$. Write $\pi(f)=g$.

Conversely, consider a homomorphism $g:F_1^-(X)\rightarrow Y$, in the diagram
$$\xymatrix@C=0.5cm{
  &X_1\ar@{-->}[d]^{f_1} \ar[r]^-{X_{1,out}} & \oplus_{k=2}^{n} B_{1k}\otimes X_k \ar[r]\ar[d]^{\oplus id\otimes g_k} & {Cok}(X_{1,out})\ar[d]^{g_1}\ar[r]&0\\
 0\ar[r]&{Ker}(Y_{1,in})\ar[r] & \oplus_{k=2}^{n} B_{1k}\otimes Y_k\ar[r]^-{Y_{1,in}} & Y_1&
  }$$
the right square commutes. There is a unique map $f_1$ such that the left square commutes.

Similarly as above, let $f_j:=g_j$ for $2\leq j\leq n$ then we get a homomorphism $f:F_1^-(X)\rightarrow Y$ corresponding to the given $g$. Write $\tau(g)=f$.

 For a morphism $f\in {Hom}_{S_1(\Lambda)}(X,F_1^+(Y))$, there exists a unique $\pi(f)$ makes the right square above commutes. And for $\pi(f)$, there exists a unique $\tau\pi(f)$ makes the left square above commutes. Since $f$ makes the left square above commutes, we have $\tau\pi(f)=f$. Similarly, $\pi\tau(g)=g$.

 Therefore, $\pi$ and $\tau$ are mutual-inverse, and then  we get a functorial isomorphism $${Hom}_\Lambda(F_1^-(X),Y)\cong{Hom}_{S_1(\Lambda)}(X,F_1^+(Y)).$$
\end{proof}

\begin{Lem}
For a Frobenius-type triangular matrix algebra $\Lambda$, there is a short exact sequence of $\Lambda$-$\Lambda$-bimodules
$$P_\bullet:\xymatrix@C=0.5cm{
  0 \rightarrow \oplus_{1\leq i<j\leq n}\Lambda e_i\otimes_{A_i}B_{ij}\otimes_{A_j}e_j\Lambda \stackrel{d}\rightarrow \oplus_{k=1}^n\Lambda e_k\otimes_{A_k}e_k\Lambda \stackrel{mult}\rightarrow\Lambda \rightarrow 0}$$
where $d$ satisfies $d(p\otimes_{A_i} b\otimes_{A_j} q):=pb\otimes_{A_j} q-p\otimes_{A_i} bq$
and the morphism ``$mult$" is given by the multiplication of $\Lambda$.

\end{Lem}
\begin{proof} Trivially,  $d$ is injective and $mult$ is surjective. Also, ${Im}(d)={Ker}(mult)$.
\end{proof}
\begin{Cor}\label{projresolution}
For a Frobenius-type triangular matrix algebra $\Lambda$ and $X\in  rep_{l.f.}(\Lambda)$, there is a projective resolution of $X$:
$$P_\bullet\otimes_\Lambda X: \xymatrix@C=0.5cm{
  0 \rightarrow \oplus_{1\leq i<j\leq n}P_i\otimes_{A_i}B_{ij}\otimes_{A_j} X_j \stackrel{d\otimes X}\rightarrow \oplus_{k=1}^n P_k\otimes_{A_k} X_k \stackrel{mult}\rightarrow X \rightarrow 0}$$
with $(d\otimes X)(p\otimes_{A_i} b\otimes_{A_j} x)=pb\otimes_{A_j} x-p\otimes_{A_i} \phi_{ij}(b\otimes_{A_j} x).$
\end{Cor}
\begin{proof} Here $P_i$ is just defined in (\ref{proj}). Then, $P_\bullet\otimes_\Lambda X$ is always exact. Since $X$ is locally free, $e_k\Lambda\otimes_\Lambda X=e_kX$ are free $A_k$-modules. Thus $P_\bullet\otimes_\Lambda X$ is a projective resolution.
\end{proof}

Following \cite{G} and \cite{GLS}, we define a functor  $T$ of ${rep}(\Lambda)$ satisfying that $$TX=\left(
       \begin{array}{c}
        X_1 \\
         X_2\\
         \vdots \\
         X_n\\
       \end{array}
     \right)_{\psi_{ij}} \;\text{ for any} \;X=\left(
       \begin{array}{c}
        X_1 \\
         X_2\\
         \vdots \\
         X_n\\
       \end{array}
     \right)_{\phi_{ij}}\in{rep}(\Lambda).$$
where $\psi_{ij}(b\otimes_j x)=-\phi_{ij}(b\otimes_j x)$ for $b\in B_{ij},~x\in X_j,~1\leq i<j\leq n$.
Obviously, $T$ is an automorphism functor.

\begin{Thm}\label{reflection functor}
For a Frobenius-type triangular matrix algebra $\Lambda$ and each $X\in  rep_{l.f.}(\Lambda)$, there are functorial isomorphisms: $$TC^+(X)\cong \tau(X)  ~~~{and}~~~TC^-(X)\cong\tau^-(X).$$
\end{Thm}

This theorem and its preparation above follow the conclusion and method of \cite{G} and \cite{GLS}. But, here $\Lambda$ is a Frobenius-type triangular matrix algebra, which includes  the classes of algebras in \cite{G} and \cite{GLS}. We need to overcome the different key point such as the different expression form of algebras by using the dual basis in Lemma \ref{dualbasis}.

Firstly, from $\Lambda$, we construct a new Frobenius-type triangular matrix algebra:

 $\tilde{\Lambda}=\left(
                                                                          \begin{array}{cccc}
                                                                            \tilde{A}_1 & \tilde{A}_{1,2} & \ldots& \tilde{A}_{1,2n} \\
                                                                            0 & \tilde{A}_2& \ldots & \tilde{A}_{2,2n} \\
                                                                             \vdots& \vdots& \ddots&
                                                                             \vdots \\
                                                                             0& 0& \cdots&
                                                                             \tilde{A}_{2n} \\
                                                                          \end{array}
                                                                        \right)$, where
                                                                        $\tilde{A}_i=\left\{
  \begin{array}{ll}
    A_i,  &\text{if}~ 1\leq i\leq n, \\
    A_{i-n}, &\text{if}~ n+1\leq i\leq 2n,
  \end{array}
\right.$ \\and \;
$\tilde{B}_{ij}=\left\{
  \begin{array}{lll}
    B_{ij},  &\text{if}~ 1\leq i<j\leq n ~\hbox{or}~ n+1\leq i<j\leq 2n\\
    {Hom}_{A_i}(B_{j-n,i},A_i), &\text{if}~ 1\leq j-n <i\leq n\\
    0, &\hbox{otherwise}
  \end{array}
\right.$\\ and\; $\tilde{A}_{ij}=\oplus_{l=0}^{j-i-1}\oplus_{i<k_1<k_2<\cdots<k_l<j}\tilde{B}_{ik_1}\otimes_{\tilde{A}_{k_1}}\tilde{B}_{k_1k_2}\otimes\cdots\otimes _{\tilde{A}_{k_l}}\tilde{B}_{k_lj}$\; for $1\leq i<j\leq 2n$.

For any non-negative integer $t$, denote $$1^{(t)}:=\Sigma_{i=t+1}^{t+n}e_i,~1_0^{(t)}:=\Sigma_{i=1}^{t+n}e_i$$ and the corresponding subalgebras $$\Lambda^{(t)}:=1^{(t)}\tilde{\Lambda}1^{(t)},~ \tilde{\Lambda}^{(t)}:=1_0^{(t)}\tilde{\Lambda}1_0^{(t)}.$$
for $0\leq t\leq n$. It is easy to see that $\tilde{\Lambda}^{(0)}\cong \Lambda^{(0)} \cong \Lambda^{(n)} \cong \Lambda,~\tilde{\Lambda}^{(n)}\cong \tilde{\Lambda}.$ And $\Lambda^{(t)}\cong S_t\cdots S_1(\Lambda)$ for $1\leq t\leq n$.

Let $X=\left(
       \begin{array}{c}
        X_1 \\
         \vdots \\
         X_{2n}\\
       \end{array}
     \right)_{\phi_{ij}}\in  rep(\tilde{\Lambda})$ satisfy the following condition:
\begin{equation}\label{outin}
X_{n+i}\stackrel{\tilde{X}_{i,out}}\longrightarrow \oplus_{k=i+1}^{n+i-1} B_{ik}\otimes X_k\stackrel{\tilde{X}_{i,in}}\longrightarrow X_i ~\hbox{such that}~ \tilde{X}_{i,in}\circ \tilde{X}_{i,out}=0 ~\hbox{for}~ 1\leq i\leq n.
\end{equation}

We define the restriction functors:
$${Res}^{(t,m)}:{rep}(\tilde{\Lambda}^{(m)}) \rightarrow  {rep}(\Lambda^{(t)}), ~~~ X \mapsto  1^{(t)}\tilde{\Lambda}^{(m)}\otimes_{\tilde{\Lambda}^{(m)}}X. $$
$${Res}_{(t,m)}:{rep}(\tilde{\Lambda}^{(m)}) \rightarrow  {rep}(\tilde{\Lambda}^{(t)}), ~~~ X \mapsto  1_0^{(t)}\tilde{\Lambda}^{(m)}\otimes_{\tilde{\Lambda}^{(m)}}X~\hbox{for}~1\leq t\leq m\leq n. $$
Obviously, ${Res}_{(t,m)}$ has a right adjoint $${Res}_{(t,m)}^*(-)={Hom}_{\tilde{\Lambda}^{(t)}}(1_0^{(t)}\tilde{\Lambda}^{(m)},-).$$

The following lemma is a generalization of \cite[Lemma 10.2]{GLS}. Their proofs are identical.
\begin{Lem}\label{res}
With the above notations, there is a  functorial isomorphisms $${Res}^{(i,i)}\circ {Res}_{(i-1,i)}^*(X)\cong F_i^+\circ {Res}^{(i-1,i-1)}(X)$$
for $X\in {rep}(\tilde{\Lambda}^{(i-1)})$ which satisfies (\ref{outin}) for all $1\leq i\leq n$.
\end{Lem}

By Lemma \ref{res} we obtain that for any $X\in {rep}(\Lambda)$ (regarded as a representation of ${rep}(\tilde{\Lambda}^{(0)})$) we have
$$\begin{array}{ll}
                       {Hom}_\Lambda(1^{(0)}\tilde{\Lambda}1^{(n)},X)  &  ={Res}^{(n,n)}\circ {Res}_{(n-1,n)}^*\circ\cdots\circ {Res}_{(0,1)}^*(X)\\
                         & =F_n^+\circ{Res}^{n-1}\circ {Res}_{(n-2,n-1)}^*\circ\cdots\circ {Res}_{(0,1)}^*(X) \\
                         & =\cdots \\
                         & =F_n^+\circ F_{n-1}^+\circ\cdots F_1^+\circ{Res}^0(X)\\
                         & =C^+(X).
                      \end{array}$$
Denote $${Res}_0={Res}_{(0,1)}\circ{Res}_{(1,2)}\circ\cdots\circ{Res}_{(n-1,n)},\;\;\;\;\;\;{Res}_0^*={Res}_{(n-1,n)}^*\circ\cdots\circ {Res}_{(1,2)}^*\circ {Res}_{(0,1)}^*.$$
It is easy to see that ${Res}_0^*$ is right adjoint of ${Res}_0$.

Now, following \cite{G} and \cite{GLS},  we construct another functor ${R}_0^*:{rep}(\Lambda^{(0)})\rightarrow {rep}(\tilde{\Lambda})$, and then show that ${R}_0^*$ is right adjoint to ${Res}_0$.

Let $X\in{rep}(\Lambda^{(0)})$. We first define $\tilde{X}\in {rep}(\tilde{\Lambda})$ by requiring that
$$\begin{array}{ll}
      {Res}^{(0,n)}(\tilde{X}) & =X. \\
        {Res}^{(n,n)}(\tilde{X})& =\oplus_{1\leq k<t\leq n}{Hom}_{A_t}(\tilde{B}_{t,n+k}\otimes e_{n+k}\Lambda^{(n)},X_t).
 \end{array}
$$
For $1\leq i<j\leq n$, it remains to define the structure map of $\tilde{X}$ as $A_j$-module morphisms:
$$\phi_{j,n+i}:\tilde{B}_{j,n+i}\otimes \tilde{X}_{n+i}\rightarrow \tilde{X}_j=X_j$$
which is given by the following composition:
$$\begin{array}{l}
      \tilde{B}_{j,n+i}\otimes_{A_i}(\oplus_{1\leq k<t\leq n}{Hom}_{A_t}(\tilde{B}_{t,n+k}\otimes e_{n+k}\Lambda^{(n)}e_{n+i},X_t)) \\
      \stackrel{{proj.}}\longrightarrow\tilde{B}_{j,n+i}\otimes_{A_i}{Hom}_{A_j}(\tilde{B}_{j,n+i}\otimes e_{n+i}\Lambda^{(n)}e_{n+i},X_j)\\
      =\tilde{B}_{j,n+i}\otimes_{A_i}{Hom}_{A_j}(\tilde{B}_{j,n+i},X_j)\\
      \stackrel{{eval.}}\longrightarrow X_j
 \end{array}
$$
where the first map is the projection on the direct summand indexed by $i,j$ and the second map is the evaluation $b\otimes \varphi=\varphi(b)$.

Secondly, we define a $\tilde{\Lambda}$-subrepresentation $R_0^*(X)$ of $\tilde{X}$ as follows. We set
$$(R_0^*(X))_i=\tilde{X}_i=X_i,~~~~(1\leq i\leq n),$$
and we define $R_0^*(X)_{n+i}$ as the subspace of $\tilde{X}_{n+i}$ generated by all of
$$(\mu_{k,t}^i)_{1\leq k<t\leq n}\in \oplus_{1\leq k<t\leq n}{Hom}_{A_t}(\tilde{B}_{t,n+k}\otimes e_{n+k}\Lambda^{(n)}e_{n+i},X_t)$$
such that, for all $1\leq t\leq n$ and $\lambda \in e_{n+t}\Lambda^{(n)} e_{n+i}$ the following relation holds:
\begin{equation}\label{subspace}
\Sigma_{1\leq k<t,\ell\in L_{kt}}\mu_{k,t}^i(\ell^*\otimes \ell\lambda)+\Sigma_{t<m\leq n,r\in R_{tm}}\phi_{tm}(r\otimes\mu_{t,m}^i(r^*\otimes \lambda))=0.
\end{equation}

  For $\mu^j=(\mu_{k,t}^j)_{1\leq k<t\leq n}\in (R_0^*(X))_{n+j},~~~1\leq j\leq n$, we deduce from the definitions that
$$\begin{array}{ll}
&\tilde{X}_{j,in}\circ \tilde{X}_{j,out}(\mu^j)\\
   =& \tilde{X}_{j,in}(\Sigma_{1\leq i<j,\ell\in L_{ij}}(\ell^*\otimes \phi_{n+i,n+j}(\ell\otimes \mu^j))
    + \Sigma_{j<k\leq n,r\in R_{ij}}(r\otimes \phi_{k,n+j}(r^*\otimes \mu^j)))\\
  =& \Sigma_{1\leq i<j,\ell\in L_{ij}}\phi_{j,n+i}(\ell^*\otimes \phi_{n+i,n+j}(\ell\otimes \mu^j))
    + \Sigma_{j<k\leq n,r\in R_{ij}}\phi_{j,k}(r\otimes \phi_{k,n+j}(r^*\otimes \mu^j))\\
    =&\Sigma_{1\leq i<j,\ell\in L_{ij}}\mu_{i,j}^j(\ell^*\otimes \ell\otimes e_{n+j})
    + \Sigma_{j<k\leq n,r\in R_{ij}}\phi_{j,k}(r\otimes \mu_{j,k}^j(r^*\otimes e_{n+j}))=0
  \end{array}.$$
So, $R_0^*(X)$ satisfies (\ref{outin}).

Thus, we have obtained a functor ${R}_0^*: {rep}(\Lambda^{(0)})\rightarrow {rep}(\tilde{\Lambda});~X\mapsto {R}_0^*(X).$  We have that ${R}_0^*$ is isomorphic to ${Res}_0^*$. This is according to the uniqueness of adjoint functors up to natural equivalence and the following lemma.
\begin{Lem}\label{adjoint}
${R}_0^*$ is right adjoint to ${Res}_0$.
\end{Lem}

It is a generalization of \cite[Lemma 10.3]{GLS}. Their proofs are identical.

\begin{Lem}\label{dual}
Let $A_i$ be a Frobenius algebra, $U$ and $V$ be finite generated free $A_i$-modules, then we have an isomorphism $D{Hom}_{A_i}(U,V)\cong{Hom}_{A_i}(V,U)$, where $D:=Hom_k(-,k)$.
\end{Lem}
\begin{proof}
Since $A_i$ is a Frobenius algebra, $DA_i\cong A_i$. Let $U=\oplus_{j=1}^nA_i$ and $V=\oplus_{t=1}^mA_i$, then $D{Hom}_{A_i}(U,V)\cong \oplus_{j=1}^n\oplus_{t=1}^mD{Hom}_{A_i}(A_i,A_i)\cong \oplus_{j=1}^n\oplus_{t=1}^m{Hom}_{A_i}(A_i,A_i)\cong{Hom}_{A_i}(V,U).$
\end{proof}

\begin{Prop}\label{restiction}
For a Frobenius-type triangular matrix algebra $\Lambda$ and $X\in {rep}_{l.f.}(\Lambda)$, there is an isomorphism $$\tau(TM)\cong{Res}^{(n,n)}\circ{R}_0^*(X),$$
where $\Lambda$ is identified  with $\Lambda^{(0)}$ and $\Lambda^{(n)}$ by their definitions.
\end{Prop}
\begin{proof}
The proof is similar to \cite[Proposition 10.4]{GLS}. The only illustration we need to add is the fact that for a locally free $\Lambda$-module $X$, we have $DHom_{A_i}(e_i\Lambda, X_i)\cong Hom_{A_i}(X_i, e_i\Lambda)$. This follows from Lemma \ref{dual}.
\end{proof}

\textbf{Proof of Theorem \ref{reflection functor}:}
By Proposition \ref{restiction}, if $X\in {rep}_{l.f.}(\Lambda)$ we have
$$TC^+(X)\cong{Res}^{(n,n)}\circ{Res}_0^*(TX)\cong{Res}^{(n,n)}\circ{R}_0^*(TX)\cong\tau(T^2X)\cong\tau(X). $$

Let $X,Y\in {rep}_{l.f.}(\Lambda)$, then we have
$$\begin{array}{ll}
    {Hom}_\Lambda(\tau^-(X),Y)  &\cong{Hom}_\Lambda(X,\tau(Y)) \\
     & \cong{Hom}_\Lambda(X,C^+(TY)) \\
     & \cong{Hom}_\Lambda(C^-(TX),Y).
  \end{array}$$
The first isomorphism is obtained by Corollary \ref{projdim} and the third isomorphism follows from Proposition \ref{adj}.

There exists a functorial isomorphism of right $\Lambda$-modules $DX\cong Hom_\Lambda(X,D\Lambda)$ for all $\Lambda$-modules $X$. Since $D\Lambda\in rep_{l.f.}(\Lambda)$, taking $Y=D\Lambda$ we get $\tau^-(X)\cong C^-(TX)$.
\qed\\

Recall that the category of Gorenstein-projective modules of $\Lambda$ is
$$\mathcal{GP}(\Lambda)=\{X\in {rep}(\Lambda)|{Ext}_\Lambda^1(X,\Lambda)\}=0.$$
As an immediate consequence of Theorem \ref{reflection functor} and the definition of $C^+(-)$ we get the following result.
\begin{Cor}
For a Frobenius-type triangular matrix algebra $\Lambda$ and $X\in{rep}_{l.f.}(\Lambda)$, the following are equivalent:

 $(i)~~X\in \mathcal{GP}(\Lambda);$

$(ii)~C^+(X)=0;$

$(iii)X_{i,in}$ is injective for all $1\leq i\leq n$.
\end{Cor}

This result has been proved in a more general case in \cite{LY}.

\section{Root systems in case of Dynkin type}

For a Frobenius-type triangular matrix algebra $\Lambda$, let $C=(c_{ij})\in M_n(\mathbb{Z})$, where $$c_{ij}=\left\{
  \begin{array}{lll}
    2,  & \hbox{if} ~i=j\\
    -\hbox{rank}_{A_i}(B_{ij}) &\hbox{if} ~ i<j\\
    -\hbox{rank}_{A_i}(B_{ji})  &\hbox{if} ~ i>j.
  \end{array}
\right.$$ Denote $c_i={dim}_k(A_i),$ then it is easy to get that $c_i c_{ij}=c_j c_{ji}=-dim_k(B_{ij})$, which means that $C$ is a symmetrizable Cartan matrix.

 Define a quadratic form $q_C:\mathbb{Z}^n\rightarrow \mathbb{Z}$ of $C$ satisfying for $x=(x_1,\cdots,x_n)^t\in \mathbb Z^n$,
\begin{equation}\label{matrixquadratic}
q_C(x):=\sum_{i=1}^nc_ix_i^2-\sum_{i<j}c_i|c_{ij}|x_ix_j.
\end{equation}

 The Cartan matrix $C$ is said to be of {\em Dynkin type} (resp. {\em Euclidean type}) if $q_C$ is positive define (resp. positive semidefinite).

We define the valued quiver $\Gamma(\Lambda)$ via the Cartan matrix $C$ whose vertices are $1,\ldots,n$ and whose arrow $i\leftarrow j$ from $j$ to $i$ exists for each pair $(i,j)$ with $i<j$ if $c_{ij}<0$, with  valuation $(-c_{ji},-c_{ij})$ on the arrow $i\leftarrow j$.
 It is well-known, , see \cite[Theorem 4.8]{Kac}, that
 \begin{Fac}\label{fact}
 $C$ is of Dynkin type if and only if $\Gamma(\Lambda)$  is a disjoint union of quivers whose underlying valued graph is of Dynkin type.
 \end{Fac}

 The standard basis vector of $\mathds{Z}^n$, denoted as $\alpha_1,\cdots, \alpha_n$,  are the positive simple roots of the Kac-Moody algebra $\mathfrak{g}(C)$ associated with $C$, that is, of the quadratic form $q_C(x)$. For $1\leq i,j\leq n$, define the reflections $s_i$ satisfying that $s_i(\alpha_j):=\alpha_j-c_{ij}\alpha_i.$

The \emph{Weyl group} $W(C)$ of $\mathfrak{g}(C)$ is the subgroup of ${Aut}(\mathds{Z}^n)$ generated by $s_1,\cdots,s_n$. It is well-known that $W(C)$ is finite if and only if $C$ is of Dynkin type.

Let $\Delta(C)$ be the sets of roots of $C$ and $\Delta_{re}(C):=\cup_{i=1}^nW(\alpha_i)$ be the set of \emph{real roots} of $C$. Let $\Delta^+(C):=\Delta(C)\cap \mathbb{N}^n$ and $\Delta_{re}^+(C):=\Delta_{re}(C)\cap \mathbb{N}^n$. The following are equivalent:

$(i)~~C$ is of Dynkin type;

$(ii)~\Delta(C)$ is finite;

$(iii)\Delta_{re}(C)=\Delta(C).$

Let $$\beta_k=\left\{
  \begin{array}{ll}
    \alpha_1,  & \hbox{if} ~k=1,\\
    s_1s_2\cdots s_{k-1}(\alpha_k) &\hbox{if} ~2\leq k\leq n
  \end{array}
\right.$$ and $$\gamma_k=\left\{
  \begin{array}{ll}
    s_ns_{n-1}\cdots s_{k+1}(\alpha_k) &\hbox{if} ~1\leq k\leq n-1,\\
     \alpha_n,  & \hbox{if} ~k=n.
  \end{array}
\right.$$

Let $c^+=s_ns_{n-1}\cdots s_1:\mathbb{Z}^n\rightarrow \mathbb{Z}^n$ and $c^-=s_1s_2\cdots s_n:\mathbb{Z}^n\rightarrow \mathbb{Z}^n$ be the \emph{Coxeter transformations}. For $k\in \mathbb{Z}$,  set $c^k:=\left\{
  \begin{array}{lll}
    (c^+)^k  & \hbox{if} ~k>0,\\
     (c^-)^{-k}  & \hbox{if} ~k<0,\\
    id &\hbox{if} ~k=0.
  \end{array}
\right.$
It follows that $c^+(\beta_k)=-\gamma_k$.

The following three lemmas are well-known, for example see \cite[Chapter VII]{IAssem}.
\begin{Lem}\label{rootinfinite}
Suppose $C$ is not of Dynkin type. Then the element $c^{-r}(\beta_i)$ and $c^s(\gamma_j)$ with $r,s\geq 0$ and $1\leq i,j \leq n$ are pairwise different elements in $\Delta_{re}^+(C)$.
\end{Lem}

Let $C$ be of Dynkin type. Let $p_i\geq 1$ be minimal with $c^{-p_i}(\beta_i)\notin \mathbb{N}^n$ for $1\leq i\leq n$, and let $q_j\geq 1$ be minimal with $c^{q_j}(\gamma_j)\notin \mathbb{N}^n$ for $1\leq j\leq n$. The elements $c^{-r}(\beta_i)$ with $1\leq i\leq n$ and $0\leq r\leq p_i-1$ are pairwise different, and the elements $c^s(\gamma_j)$ with $1\leq j\leq n$ and $0\leq s\leq q_j-1$ are pairwise different.

\begin{Lem}\label{rootfinite}
Assume that $C$ is of Dynkin type, then $\Delta^+(C)=\{c^{-r}(\beta_i)|1\leq i\leq n,~0\leq r\leq p_i-1\}=\{c^s(\gamma_j)|1\leq j\leq n,~0\leq s\leq q_j-1\}.$
\end{Lem}

\begin{Lem}\label{positvenegative}
Assume that $C$ is of Dynkin type. For every positive vector $\underline{x}$, there exist $s\geq 0$ such that $c^s \underline{x}>0$ but $c^{s+1} \underline{x}\ngtr 0$, and $t\geq0$ such that $c^{-t} \underline{x}>0$ but $c^{-t-1} \underline{x}\ngtr 0$.
\end{Lem}




\begin{Def}
For a Frobenius-type triangular matrix algebra $\Lambda$, let $X$ be a locally free $\Lambda$-module. Let $r_i$ be the rank of the free $A_i$-module $X_i$. Following \cite{GLS} we denote $$\underline{rank}(X):=(r_1,\cdots,r_n).$$
\end{Def}

\begin{Lem}\label{rankroot}
 $\underline{rank}(P_k)=\beta_k,~\underline{rank}(I_k)=\gamma_k$.
\end{Lem}

This lemma is a generalization of \cite[Lemma 3.2 and 3.3]{GLS}. Its proof is similar to the proof of \cite[Lemma 3.2]{GLS}, using of the resolution (\ref{rankproj}) and (\ref{injresolution}).

\begin{Def}
  For a Frobenius-type triangular matrix algebra $\Lambda$, following \cite{GLS}, a $\Lambda$-module $X$ is called {\em $X~\tau$-locally free} if  $\tau^k(X)$ is locally free for all $k\in \mathbb{Z}$. Moreover,  $X$ is called  {\em indecomposable $\tau$-locally free} if $X$ cannot  be written as a sum of two proper $\tau$-locally free $\Lambda$-modules.
\end{Def}

The following proposition is a generalization of \cite[Proposition 9.6 and 11.4]{GLS}.
\begin{Prop}\label{taulf}
For a Frobenius-type triangular matrix algebra $\Lambda$, let $X$ be a rigid and locally free $\Lambda$-module, \\
$(i)~ F_1^+(X)$ is a rigid and locally free $S_1(\Lambda)$-module and $F_n^-$ is a rigid and locally free $S_{n-i}\cdots S_1(\Lambda)$-module;\\
$(ii)~ X$ is $\tau$-locally free and $\tau^k(X)$ is rigid for all $k\in \mathbb{Z}$.
\end{Prop}
\begin{proof}
  Its proof is identical to the proof of \cite[Proposition 9.6 and 11.4]{GLS}. The only illustration we need to add is the fact that for a locally free $\Lambda$-module $X$ and its minimal projective resolution of the form
$$0\rightarrow P''\rightarrow P'\rightarrow X\rightarrow 0,$$
we have that $P'$ and $P''$ are two direct sums of $P_i$ for $1\leq i \leq n$. This follows from Proposition \ref{dimension1}, Corollary \ref{projdim} and the horseshoe lemma.
\end{proof}

\begin{Lem}\label{ranks1}
 For a Frobenius-type triangular matrix algebra $\Lambda$, let $X$ be an indecomposable $\tau$-locally free $\Lambda$-module. Then $X$ is isomorphic to $E_1$ if and only if $F_1^+(X)=0$ (or equivalently, $s_1(\underline{rank}(X))\ngtr 0).$ If $X\ncong E_1,$ then $F_1^+(X)$ is an indecomposable $\tau$-locally free $S_1(\Lambda)$-module and $\underline{rank}(F_1^+(X))=s_1(\underline{rank}(X))$.
\end{Lem}
\begin{proof}
Since $X$ is indecomposable $\tau$-locally free, we obtain that $X_{1,in}$ is surjective. If $X\ncong E_1,$ we have an exact sequence
$$\xymatrix@C=0.5cm{
  0 \ar[r] & {Ker}(X_{1,in}) \ar[r]  & \oplus_{k=2}^nB_{1k}\otimes X_k \ar[r] & X_1 \ar[r] & 0 }$$
So $$(\underline{rank}(F_1^+(X)))_1=\sum_{k=2}^n |c_{1k}|a_k-a_1=(s_1(\underline{rank}(X)))_1.$$

If $X\cong E_1,$ we have $F_1^+(X)=0$ by the definition of $F_1^+$. And $s_1(\underline{rank}(E_1))_1=-1$.
\end{proof}

We need the following two propositions, which are generalizations of \cite[Proposition 11.5 and 11.6]{GLS} respectively. Their proofs are similar.
\begin{Prop}\label{ranktau}
For a Frobenius-type triangular matrix algebra $\Lambda$ and  a indecomposable $\tau$-locally free $\Lambda$-module $X$, the following statements hold:

$(i)$ If $\tau^k(X)\neq 0$ for some $k\in \mathbb{Z}$, then $\underline{rank}(\tau^k(X))=c^k(\underline{rank}(X))$;

$(ii)$If $\tau^k(X)\neq 0$ for some $k\in \mathbb{Z}$ and $\underline{rank}(X)$ is contained in $\Delta^+(C)$, then $\underline{rank}(\tau^k(X))$ is in $\Delta^+(C)$.
\end{Prop}

\begin{Prop}\label{tauroot}
For a Frobenius-type triangular matrix algebra $\Lambda$ and a $\Lambda$-module $X$, if either $X\cong \tau^{-k}(P_i)$ or $X\cong \tau^k(I_i)$ for some $k\geq 0$ and $1\leq i\leq n$, the following statements hold:

$(i)~~X$ is $\tau$-locally free and rigid;

$(ii)~~\underline{rank}(X)\in \Delta_{re}^+(C)$;

$(iii)~~$If either $Y\cong \tau^{-m}(P_j)$ or $Y\cong \tau^m(I_j)$ for some $m\geq 0$ and $1\leq j\leq n$ with $\underline{rank}(X)=\underline{rank}(Y)$, then $X\cong Y$.
\end{Prop}

\begin{Thm} \label{mainone} For a Frobenius-type triangular matrix algebra $\Lambda$, the following statements hold:

(a) The number of isomorphism classes of indecomposable $\tau$-locally free $\Lambda$-modules is finite if and only if $C$ is of Dynkin type.

(b) If $C$ is of Dynkin type, then the mapping $\underline{rank}:X\mapsto \underline{rank}(X)$ induces a bijection between the set of isomorphism classes of indecomposable $\tau$-locally free $\Lambda$-modules and the set of positive roots of the quadratic form $q_C(x)$.
\end{Thm}
\begin{proof}
 In the case that $C$ is not of Dynkin type, by Lemma \ref{rootinfinite}, Lemma \ref{rankroot} and Proposition \ref{ranktau}, we know there are infinite many isomorphism classes of indecomposable $\tau$-locally free $\Lambda$-modules.

In the case that $C$ is of Dynkin type, firstly, we need to prove that $\underline{rank}(X)\in \Delta^+(C)$ for any indecomposable $\tau$-locally free $\Lambda$-module $X$. We denote $\underline{rank}(X)=\underline{x}$. By Lemma \ref{positvenegative} there exists a least $s$ such that $c^s \underline{x}>0$ but $c^{s+1} \underline{x}\ngtr 0$. Because $c^+=s_n\cdots s_1,$ there also exists a least $i$ such that $0\leq i \leq n-1,~s_i\cdots s_1c^t\underline{x}>0,$ but $s_{i+1}\cdots s_1c^t\underline{x}\ngtr0$.

We know that $X'=F_i^+\cdots F_1^+C^{+t}X$ is indecomposable $\tau$-locally free by Lemma \ref{ranks1} and that
$$\underline{rank}(F_i^+\cdots F_1^+C^{+t}X)=s_i\cdots s_1c^t\underline{x}.$$
Because $s_{i+1}(\underline{rank}(X'))\ngtr 0$, there is an isomorphism $X'\cong E_{i+1}$ by Lemma \ref{ranks1}. So $s_i\cdots s_1c^t\underline{x}=\alpha_{i+1}$, and according to Lemma \ref{rootfinite} the vector $\underline{x}=c^{-t}s_1\cdots s_i\alpha_{i+1}=c^{-t}\beta_{i+1}$ is a positive root of $q_C(x)$ and $\underline{rank}(-)$ is surjective.

If $X$ and $Y$ are indecomposable $\tau$-locally free $\Lambda$-modules such that $\underline{rank}(X)=\underline{rank}(Y)$, then we have, as earlier $$F_i^+\cdots F_1^+C^{+t}X\cong E_{i+1}\cong F_i^+\cdots F_1^+C^{+t}Y$$ so that $$X\cong C^{-t}F_1^-\cdots F_i^-E_{i+1}\cong Y.$$ Thus $\underline{rank}(-)$ is an injective mapping.
\end{proof}

\begin{Cor} For the path algebra $\Lambda=AQ$ of an acyclic quiver $Q$ over a Frobenius algebra $A$, the following statements hold:

(a)\;  The number of isomorphism classes of indecomposable $\tau$-locally free $\Lambda$-modules is finite if and only if $Q$ is of Dynkin type;

(b)\; In the case when $Q$ is of Dynkin type,  the mapping $\underline{rank}:X\mapsto \underline{rank}(X)$ induces a bijection between the set of isomorphism classes of indecomposable $\tau$-locally free $\Lambda$-modules and the set of positive roots of the quadratic form $q_Q(x)=\sum_{i\in Q_0}x_i^2-\sum_{\alpha\in Q_1}x_{s(\alpha)}x_{t(\alpha)}$, where $x=(x_1,\ldots,x_n)^t\in \mathbb{Z}^n$.
\end{Cor}
\begin{proof} By Remark \ref{rem1.2} $(iii)$,  $\Lambda=AQ$  is a Frobenius-type triangular matrix algebra via  taking all $A_i=A$ and $B_{ij}=\oplus_{s=1}^{\#\{\alpha:j\rightarrow i\}} A$. So, $c_i=c_j=dim_kA$ and then $$-c_{ij}=-c_{ji}=dim_kB_{ij}/dim_kA=\#\{\alpha:j\rightarrow i\}$$ which means that the Cartan matrix $C$ is symmetric and in its corresponding valued quiver $\Gamma(\Lambda)$, the valuation $(-c_{ij}, -c_{ji})$ is given with the number of arrows from $j$ to $i$ in $Q$. By the definitions of the quadratic forms, we have:
$$q_C(x)=\sum_{i=1}^nc_ix_i^2-\sum_{i<j}c_i|c_{ij}|x_ix_j=dim_kA(\sum_{i=1}^nx_i^2-\sum_{i<j}|c_{ij}|x_ix_j)=q_Q(x)dim_kA.$$
Hence, the positivity definite property of  $q_C$ and $q_Q$ are the same with each other.
 Thus, the statements (a) and (b) follow respectively from Theorem \ref{mainone} (a) and (b).
\end{proof}

Recall in \cite{L} that for a generalized path algebra $\Lambda=k(Q,\mathcal A)$, there is a corresponding valued quiver $\Upsilon(\Lambda)$, whose the set of  vertices $\Upsilon(\Lambda)_0=Q_0$ and give an arrow $i\leftarrow j$ of $\Upsilon(\Lambda)$, whose valuation is $(d_{ji},d_{ij})$ with   $d_{ji}=|Q_{ij}|{dim}_kA_i$ and $d_{ij}=|Q_{ij}|{dim}_kA_j$ where $|Q_{ij}|$ means the number of arrows from $j$ to $i$ in $Q$.

In the case when $Q$ is acyclic, in order to realize $k(Q,\mathcal A)$ as a triangular matrix algebra,  we re-arrange the order of vertices in $Q$ via assuming $i<j$ if there exists a path from $j$ to $i$. Let $B_{ij}=A_i Q_{ij}A_j$ for $Q_{ij}$ the set of arrows from $j$ to $i$ in $Q$ and then define $A_{ij}$ as in Definition \ref{frobenius type} and put $A_i$ at the $(i,i)$-array. Then we obtain the triangular matrix algebra which is equal to $\Lambda=k(Q,\mathcal A)$.

  Moreover, $-c_{ij}={rank}_{A_i}(B_{ij})=|Q_{ij}|{dim}_kA_j=d_{ij}$. Thus, it follows that the valued quiver $\Upsilon(\Lambda)$ is coincident with $\Gamma(\Lambda)$.

 Define a quadratic form $q_{k(Q,\mathcal A)}:\mathbb{Z}^n\rightarrow \mathbb{Z}$ of $k(Q,\mathcal A)$ satisfying for $x=(x_1,\cdots,x_n)^t\in \mathbb Z^n$,
\begin{equation}\label{genpathquadratic}
q_{k(Q,\mathcal A)}(x)=\sum_{i\in Q_0}d_i x_i^2-\sum_{\alpha\in Q_1}d_{s(\alpha)}d_{s(\alpha)t(\alpha)}x_{s(\alpha)}x_{t(\alpha)}.
\end{equation}

Comparing this quadratic form with that in (\ref{matrixquadratic}), it is easy to see that $q_{k(Q,\mathcal A)}$ is the special case of $q_C$ for $\Lambda=k(Q,\mathcal A)$.

\begin{Cor}\label{newversion}  For an acyclic quiver $Q$ and its generalized path algebra $\Lambda=k(Q,\mathcal A)$ endowed by Frobenius algebras $A_i$ at all  vertices $i\in Q_0$, the following statements hold:

 (a)\; The number of isomorphism classes of indecomposable $\tau$-locally free $\Lambda$-modules is finite if and only if $\Upsilon(\Lambda)$ is of Dynkin type.

(b) If $\Omega(\Lambda)$ is of Dynkin type, then the mapping $\underline{rank}:X\mapsto \underline{rank}(X)$ induces a bijection between the set of isomorphism classes of indecomposable $\tau$-locally free $\Lambda$-modules and the set of positive roots of the $q_{k(Q,\mathcal A)}(x)=\sum_{i\in Q_0}d_i x_i^2-\sum_{\alpha\in Q_1}d_{s(\alpha)}d_{t(\alpha)}x_{s(\alpha)}x_{t(\alpha)}$, where $x=(x_1,\ldots,x_n)^t\in \mathbb{Z}^n$ and $d_i={dim}(A_i)$.
\end{Cor}
\begin{proof}
They follow directly from $\Upsilon(\Lambda)=\Gamma(\Lambda)$, Fact \ref{fact}, Theorem \ref{mainone} and that $q_{k(Q,\mathcal A)}=q_C$ for $\Lambda=k(Q,\mathcal A)$.
\end{proof}

\section{Analog of APR-tilting module for $\Lambda$}
The $APR$-tilting modules were introduced by Auslander, Platzeck and Reiten in \cite{APR} to interpret $BGP$-reflection functors as homomorphism functors of certain tilting modules. Also, it was the beginning of tilting theory.

For a Frobenius-type triangular matrix algebra $\Lambda$, for the case $i=1$, we define $T_1:=\Lambda/P_1\oplus \tau^-(P_1)$ and  call $T_1$ a {\em generalized APR-tilting module} of $\Lambda$. This case follows from the fact that $i=1$ is a {\em ``sink vertex"} so as to gain the reflection functor.

\begin{Prop}\label{tilting}
For a Frobenius-type triangular matrix algebra $\Lambda$, $T_1$ is a tilting $\Lambda$-module.
\end{Prop}
\begin{proof}
Since $P_1$ is a rigid and locally free, $\tau^-(P_1)$ is rigid and locally free by Proposition \ref{taulf}. So $T_1$ is locally free. Then by Corollary  \ref{projdim}, ${proj.dim} T_1\leq 1$.  Thus, ${Ext}_\Lambda^1(T_1,T_1)\cong {DHom}_\Lambda(T_1,\tau(T_1))={DHom}_\Lambda(T_1,P_1)=0$. Because $\Lambda$ is connected, $P_1$ have no injective summand. Since $\tau^-$ take non-injective indecomposable module to non-projective indecomposable module,  $T_1:=\Lambda/P_1\oplus \tau^-(P_1)$ has the same number of summands as primitive idempotents of $\Lambda$.  So $T_1$ is a tilting $\Lambda$-module.
\end{proof}

\begin{Rem}
A similar analog of generalized APR-tilting modules was introduced in \cite{Llp} for a class of triangular matrix alegbras.
\end{Rem}

\begin{Lem}\label{endiso}
For a Frobenius-type triangular matrix algebra $\Lambda$, there is an algebra isomorphism ${End}_\Lambda(T_1)\cong S_1(\Lambda).$
\end{Lem}
\begin{proof}
Clearly, when $2\leq i,j\leq n$
$$\begin{array}{ll}
                        e'_i{End}_\Lambda (T_1) e'_j &  \cong{Hom}_{{End} (T_1)}({End}_\Lambda (T_1)e'_i,{End}_\Lambda (T_1)e'_j)\\
                         & \cong{Hom}_{{End} (T_1)}({Hom}(T_1,P_i),{Hom}(T_1,P_j)) \\
                         & \cong{Hom}_\Lambda(P_i,P_j) \\
                         & \cong e_i\Lambda e_j.
                      \end{array}.$$
When $i=1,j=1,~e'_i{End}_\Lambda (T_1) e'_j\cong {Hom}_\Lambda(\tau^-(P_1),\tau^-(P_1))\cong {Hom}_\Lambda(P_1,P_1)\cong A_1$.\\
When $i=1,j>1,~e'_i{End}_\Lambda (T_1) e'_j\cong {Hom}_\Lambda(\tau^-(P_1),P_j)=0.$\\
When $j=1,i>1,~e'_i{End}_\Lambda (T_1) e'_j\cong {Hom}_\Lambda(P_i,\tau^-(P_1))$.

Since there is an minimal injective resolution: $$0\rightarrow P_1\rightarrow I_1\rightarrow \oplus_{j=2}^n I_j\otimes B_{j1}\rightarrow 0.$$
then we have $$0\rightarrow \nu^{-}(P_1)\rightarrow P_1\rightarrow \oplus_{j=2}^n P_j\otimes B_{j1}\rightarrow \tau^-(P_1)\rightarrow 0.$$
Applying functor ${Hom}_\Lambda(P_i,-)$ for $i>1$.

Since ${Hom}_\Lambda(P_i,P_1)=0,~{Hom}_\Lambda(P_i,\oplus_{k=2}^n P_k\otimes B_{k1})\rightarrow {Hom}_\Lambda(P_i,\tau^-(P_1))$ is injective.

Since $P_i$ is projective $\Lambda$-module, ${Hom}_\Lambda(P_i,\oplus_{k=2}^n P_k\otimes B_{k1})\rightarrow {Hom}_\Lambda(P_i,\tau^-(P_1))$ is surjective.

So, ${Hom}_\Lambda(P_i,\oplus_{k=2}^n P_k\otimes B_{k1})\cong {Hom}_\Lambda(P_i,\tau^-(P_1))$.

Then, $e'_i{End}_\Lambda (T_1) e'_1\cong {Hom}_\Lambda(P_i,\oplus_{k=2}^n P_k\otimes B_{k1})\cong \oplus_{k=2}^n e_1\Lambda e_k\otimes B_{k1}\cong A_{i1}$.

At last, ${End}_\Lambda (T_1)\cong S_1(\Lambda)$.
\end{proof}

Using Proposition \ref{tilting} and Lemma \ref{endiso}, we can prove the following theorem.

\begin{Thm}
For a Frobenius-type triangular matrix algebra $\Lambda$, there is a functorial isomorphism $$F_1^+(-)\cong {Hom}_\Lambda(T_1,-): {rep}(\Lambda)\rightarrow {rep}(S_1(\Lambda)).$$
\end{Thm}
\begin{proof}
We know that $P_1\cong E_1$.

Since we have
$\nu^-(D(e_1\Lambda))={Hom}_\Lambda(D(D(e_1\Lambda)),\Lambda)\cong {Hom}_\Lambda(e_1\Lambda,\Lambda)\cong \Lambda e_1=P_1$.

Also
\begin{align}
\nu^-(D(B_{1j}\otimes_{A_j}e_j\Lambda))&={Hom}_\Lambda(B_{1j}\otimes_{A_j}e_j\Lambda,\Lambda)\\
&\cong {Hom}_{A_j}(B_{1j},{Hom}_\Lambda(e_j\Lambda,\Lambda))\\
&\cong{Hom}_{A_j}(B_{ij},\Lambda e_j)\\
&\cong \Lambda e_j\otimes_{A_j}{Hom}_{A_j}(B_{1j},A_j) \label{fingen}\\
&\cong P_j\otimes_{A_j}B_{ji}.
\end{align}
Since $\Lambda e_j$ is a finite generated projective right $A_j$-module,the  isomorphism in (\ref{fingen}) comes from \cite{modcat}.
So applying the quasi-inverse Nakayama functor $\nu^-$ to (\ref{injresolution}) we get an exact sequence
\begin{equation}\label{tau}
0\rightarrow \nu^-(E_1)\rightarrow P_1\rightarrow\oplus_{j=2}^{n}P_j\otimes_{A_j}B_{ji} \rightarrow \tau^-(P_1) \rightarrow 0.
\end{equation}
where $\theta_{1j}:P_1\rightarrow P_1\otimes B_{j1}$ is given by $\lambda e_i\mapsto \Sigma_{r\in R_{ij}}\lambda r\otimes r^*$.
We have isomorphism ${Hom}_\Lambda(P_j\otimes B_{j1},X)\cong {Hom}_\Lambda({Hom}_{A_j}(B_{ij},\Lambda e_j),X)\cong B_{1j}\otimes {Hom}_\Lambda(P_j,X)\cong B_{1j}\otimes X_j.$ The isomorphism ${Hom}_\Lambda(P_j\otimes B_{j1},X)\rightarrow B_{1j}\otimes X_j$ is given by $f\mapsto \Sigma_{r\in R_{ij}}\lambda r\otimes f(e_j\otimes r^*)$.
And isomorphism ${Hom}_\Lambda(P_1,X)\rightarrow X_1$ is given by $g\mapsto g(e_1)$.
We get a commutative diagram
\begin{equation*}
\CD
   {Hom}_\Lambda(P_j\otimes B_{j1},X) @>{Hom}_\Lambda(\theta_{1j},X)>> {Hom}_\Lambda(P_1,X) \\
    @V\eta_{1j}^X  VV @V \eta_i^X VV  \\
    B_{1j}\otimes X_j @>\varphi_{1j}>> X_1
\endCD
\end{equation*}
This follows from that for $f\in {Hom}_\Lambda(P_j\otimes B_{j1},X)$ and $r\in R_{1j}$ we have $$f(r\otimes r^*)=\varphi_{1j}(r\otimes f(e_j\otimes r^*).$$
Applying functor ${Hom}_\Lambda(-,X)$ to (\ref{tau}) for $X\in \Lambda-mod$, we obtain a commutative diagram:
\begin{equation*}
\CD
  0 @>>>{Hom}_\Lambda(\tau^-(P_1),X) @>>> \oplus_{j=2}^n {Hom}_\Lambda(P_j\otimes B_{j1},X) @>>> {Hom}_\Lambda(P_1,X) \\
  @V  VV @V  VV @V  VV @V  VV  \\
  0 @>>> {Ker}(X_{1,in}) @>>> \oplus_{j=2}^n B_{1j}\otimes X_j @>X_{1,in}>> X_1
\endCD
\end{equation*}
Since the last two terms are isomorphism, we obtain isomorphism ${Hom}_\Lambda(\tau^-(P_1),X)\cong {Ker}(M_{1,in})$.
Together with Lemma \ref{endiso}, we get the functorial isomorphism $F_1^+(-)\cong {Hom}_\Lambda(T_1,-).$
\end{proof}

This theorem is the main result in this section, whose corresponding analog in \cite{GLS} is the \cite[Theorem 9.7]{GLS}. But the method for proving in \cite{GLS} is incomplete for our case, the Frobenius-type triangular matrix algebra $\Lambda$.

 Besides Corollary \ref{newversion}, the main results in this paper, including those in this section, are interesting to be restricted two special cases, that is, $\Lambda$ is either a generalized path algebra $\Lambda=k(Q,\mathcal A)$ endowed by Frobenius algebras $A_i$ at each vertex $i\in Q_0$ or a path algebra $\Lambda=AQ$ of quiver $Q$ over a Frobenius algebra $A$.
\\

\textbf{Acknowledgements.}\; {\em We would like to thank  Christof Geiss and Jan Schr{\"o}er for their valuable comments and useful advices.
This project is supported by the National Natural Science Foundation of
China (No.11271318, No.11571173 and No.J1210038)  and the Zhejiang Provincial Natural Science
Foundation of China (No.LZ13A010001).}

\def\cprime{$'$}
\providecommand{\bysame}{\leavevmode\hbox to3em{\hrulefill}\thinspace}
\providecommand{\MR}{\relax\ifhmode\unskip\space\fi MR }
\providecommand{\MRhref}[2]{%
  \href{http://www.ams.org/mathscinet-getitem?mr=#1}{#2}
}
\providecommand{\href}[2]{#2}




\begin{thebibliography}{10}

\bibitem{AIR}
Takahide Adachi, Osamu Iyama, and Idun Reiten, \emph{{$\tau$}-tilting theory},
  Compos. Math. \textbf{150} (2014), no.~3, 415--452. \MR{3187626}

\bibitem{modcat}
Frank~W. Anderson and Kent~R. Fuller, \emph{Rings and categories of modules},
  second ed., Graduate Texts in Mathematics, vol.~13, Springer-Verlag, New
  York, 1992. \MR{1245487 (94i:16001)}

\bibitem{IAssem}
Ibrahim Assem, Daniel Simson, and Andrzej Skowro{\'n}ski, \emph{Elements of the
  representation theory of associative algebras. {V}ol. 1}, London Mathematical
  Society Student Texts, vol.~65, Cambridge University Press, Cambridge, 2006,
  Techniques of representation theory. \MR{2197389 (2006j:16020)}

\bibitem{APR}
Maurice Auslander, Mar{\'{\i}}a~In{\'e}s Platzeck, and Idun Reiten,
  \emph{Coxeter functors without diagrams}, Trans. Amer. Math. Soc.
  \textbf{250} (1979), 1--46. \MR{530043 (80c:16027)}

\bibitem{BGP}
I.~N. Bern{\v{s}}te{\u\i}n, I.~M. Gel{\cprime}fand, and V.~A. Ponomarev,
  \emph{Coxeter functors, and {G}abriel's theorem}, Uspehi Mat. Nauk
  \textbf{28} (1973), no.~2(170), 19--33. \MR{0393065 (52 \#13876)}

\bibitem{DR}
Vlastimil Dlab and Claus~Michael Ringel, \emph{On algebras of finite
  representation type}, J. Algebra \textbf{33} (1975), 306--394. \MR{0357506
  (50 \#9974)}

\bibitem{DR2}
\bysame, \emph{Indecomposable representations of graphs and algebras}, Mem.
  Amer. Math. Soc. \textbf{6} (1976), no.~173, v+57. \MR{0447344 (56 \#5657)}

\bibitem{Ga}
Peter Gabriel, \emph{Unzerlegbare {D}arstellungen. {I}}, Manuscripta Math.
  \textbf{6} (1972), 71--103; correction, ibid. 6 (1972), 309. \MR{0332887 (48
  \#11212)}

\bibitem{GaAR}
\bysame, \emph{Auslander-{R}eiten sequences and representation-finite
  algebras}, Representation theory, {I} ({P}roc. {W}orkshop, {C}arleton
  {U}niv., {O}ttawa, {O}nt., 1979), Lecture Notes in Math., vol. 831, Springer,
  Berlin, 1980, pp.~1--71. \MR{607140 (82i:16030)}

\bibitem{G}
\bysame, \emph{Auslander-{R}eiten sequences and representation-finite
  algebras}, Representation theory, {I} ({P}roc. {W}orkshop, {C}arleton
  {U}niv., {O}ttawa, {O}nt., 1979), Lecture Notes in Math., vol. 831, Springer,
  Berlin, 1980, pp.~1--71. \MR{607140 (82i:16030)}

\bibitem{GLS}
Christof Geiss, Bernard Leclerc, and Jan Schr{\"o}er, \emph{Quivers with
  relations for symmetrizable cartan matrices {I} : Foundations},
  arXiv:1410.1403, 2015, pp.~1--67.

\bibitem{IY}
Yasuo Iwanaga, \emph{On rings with finite self-injective dimension}, Comm.
  Algebra \textbf{7} (1979), no.~4, 393--414. \MR{522552 (80h:16024)}

\bibitem{Kac}
V.G. Kac, \emph{Infinite-dimensional Lie algebras}, Cambridge university press, (1994).

\bibitem{Lam}
T.Y.Lam, \emph{Serre's problem on projective modules}, Springer Science \& Business Media, (2010).

\bibitem{L}
Fang Li, \emph{Modulation and natural valued quiver of an algebra}, Pacific Journal of Mathematics (2012), 256(1): 105-128.

\bibitem{LY}
\bysame and Chang Ye, \emph{Gorenstein {P}rojective {M}odules {O}ver a {C}lass
  of {G}eneralized {M}atrix {A}lgebras and their {A}pplications}, Algebr.
  Represent. Theory \textbf{18} (2015), no.~3, 693--710. \MR{3357944}

\bibitem{Llp}
Liping Li, \emph{Stratifications of finite directed categories and generalized
  {APR} tilting modules}, Comm. Algebra \textbf{43} (2015), no.~5, 1723--1741.
  \MR{3316816}

\bibitem{Ringel}
Claus~Michael Ringel, \emph{Representations of {$K$}-species and bimodules}, J.
  Algebra \textbf{41} (1976), no.~2, 269--302. \MR{0422350 (54 \#10340)}

\bibitem{RZ}
\bysame and Pu~Zhang, \emph{Representations of quivers over the
  algebra of dual numbers}, arXiv preprint arXiv:1112.1924 (2011).

\bibitem{SL}
Meihua Shi and Fang Li, \emph{The category of modules over a triangular matrix
  ring of order {$n$} and its homological characterizations}, Acta Math. Sinica
  (Chin. Ser.) \textbf{49} (2006), no.~1, 215--224. \MR{2249410 (2007f:16013)}

\end{thebibliography}
\end{document}